\documentclass[11pt,twoside]{amsart}
\usepackage{latexsym,amssymb,amsmath}

\numberwithin{equation}{section}
\newtheorem{theorem}{Theorem}[section]
\newtheorem{lemma}[theorem]{Lemma}
\newtheorem{proposition}[theorem]{Proposition}
\newtheorem{corollary}[theorem]{Corollary}

\theoremstyle{definition}
\newtheorem{definition}[theorem]{Definition} 
 
\newtheorem{remark}[theorem]{Remark}
\newtheorem{example}[theorem]{Example}

\begin{document}

%%%%%%%%%%%%%%%%%%%%%%%%%%%%%%%%%%%%%%%%%%%%%%%%%%%%%%%%%%%%%%%%%%%%%

\title[]{Complete intersections in binomial and lattice ideals}

\thanks{The first author was partially supported by CONACyT. 
The second author was partially supported by SNI}

\author{Hiram H. L\'opez}
\address{
Departamento de
Matem\'aticas\\
Centro de Investigaci\'on y de Estudios
Avanzados del
IPN\\
Apartado Postal
14--740 \\
07000 Mexico City, D.F.
}
%\email{hlopez@math.cinvestav.mx}

\author{Rafael H. Villarreal}
\address{
Departamento de
Matem\'aticas\\
Centro de Investigaci\'on y de Estudios
Avanzados del
IPN\\
Apartado Postal
14--740 \\
07000 Mexico City, D.F.
}
\email{vila@math.cinvestav.mx}
%\urladdr{http://www.math.cinvestav.mx/$\sim$vila/}

\keywords{Complete intersections, lattice ideals, binomial ideals,
evaluation codes, monomial curves, toric ideals, vanishing ideals}
\subjclass[2010]{Primary 13F20; Secondary 14H45, 13P25, 11T71.} 

\begin{abstract} For the family of graded lattice ideals 
of dimension $1$, we establish a complete intersection criterion in 
algebraic and geometric
terms. In positive characteristic, it is shown that all ideals of
this family are 
binomial set theoretic complete intersections. In characteristic 
zero, we show that an arbitrary lattice ideal which is a binomial set
theoretic complete intersection is a complete intersection. 
\end{abstract}

\maketitle 

\section{Introduction}\label{intro-ci-vanishing}

This work was motivated in large part by the interest in two
classes of  $1$-dimensional graded lattice ideals, namely, the
family of vanishing ideals of projective algebraic toric sets over
finite fields and the family of toric ideals of monomial curves over
arbitrary fields. Our main results will apply to these families.     

Let $S=K[t_1,\ldots,t_s]$ be a positively
graded polynomial ring over a field $K$. As usual, by the {\it
dimension\/} of an ideal $L\subset S$ we mean the Krull dimension of
the quotient ring $S/L$. In this paper we study the
family of graded lattice ideals 
$L\subset S$ satisfying
that $V(L,t_i)=\{0\}$ for $i=1,\ldots,s$, where $V(L,t_i)$ is the zero-set
of the ideal $(L,t_i)$. An ideal is in this family if and only if it
is a $1$-dimensional graded lattice ideal
(Proposition~\ref{apr24-12}). 
One of our main results is a classification of
the complete 
intersection property
for this family of ideals (Theorem~\ref{ci-lattice}). Using a result 
of \cite{katsabekis-morales-thoma}, we show 
that in positive characteristic all ideals in this family are
binomial set theoretic complete intersections
(Proposition~\ref{stcib-1-dim-lattice}).  

As an application, we recover a result of \cite{stcib} that characterizes
complete intersection toric ideals of monomial curves over arbitrary fields
(Corollary~\ref{LaConcepcion-dec24-2011}). This result was used in
\cite{stcib} to give a classification of 
the complete intersection property using the notion of a binary
tree. This classification was adapted in \cite{stcib-algorithm} to
give an 
effective algorithm that
checks the complete intersection property. Toric ideals of monomial
space curves were 
first studied 
by Herzog \cite{He3}. They have been studied by many
authors \cite{ARZ,stcib-algorithm,stcib,delorme,Eli1,ElVi,moh,thoma}.

If ${\rm char}(K)=0$, using commutative algebra, 
we show that an arbitrary lattice ideal which is a binomial set theoretic
complete intersection is a complete intersection (Theorem~\ref{adelantado1}).
 This result complements \cite[Corollary~3.10]{morales-thoma}, where the lattice
ideal is taken with respect to 
a partial character  but the lattice ideal is assumed to be 
positive. Another result in this area shows that any primary binomial ideal
over a field of characteristic zero is radical
\cite[Proposition~2.3]{stcib}. This type of results were inspired by 
\cite[Theorem~4]{BMT}, where it was shown that if a toric ideal is a 
binomial set theoretic complete intersection and
${\rm char}(K)=0$, then it is a complete intersection. 

Complete intersection lattice ideals have been characterized in
\cite{FMS,katsabekis-morales-thoma,morales-thoma}, 
in terms of semigroup
gluing, mixed dominating matrices and polyhedral geometry.  Since our
methods of proof are 
entirely different to those
of \cite{FMS,katsabekis-morales-thoma,morales-thoma}, we hope that
our approach can be used to 
examine some other problems in the area from another perspective. 

The rest of this paper is about applications of our results
to vanishing ideals over finite
fields. Vanishing ideals are connected to coding theory as is seen
below.

Let $\mathbb{F}_q$  be a finite field with $q$ elements and 
let $v_1,\ldots,v_s$ be a sequence of vectors in
$\mathbb{N}^n$. Consider the {\it projective algebraic toric
set\/} 
$$
X:=\{[(x_1^{v_{11}}\cdots x_n^{v_{1n}},\ldots,x_1^{v_{s1}}\cdots
x_n^{v_{sn}})]\, \vert\, x_i\in \mathbb{F}_q^*\mbox{ for all
}i\}\subset\mathbb{P}^{s-1}
$$
parameterized by monomials, where $v_i=(v_{i1},\ldots,v_{in})$, 
$\mathbb{F}_q^*=\mathbb{F}_q\setminus\{0\}$ and 
$\mathbb{P}^{s-1}$ is a projective space over the field $\mathbb{F}_q$. The set
$X$ is a multiplicative group under componentwise multiplication. 

Let $S=\mathbb{F}_q[t_1,\ldots,t_s]=\oplus_{d=0}^\infty S_d$ be a
polynomial ring over the field $\mathbb{F}_q$ with the standard
grading. The {\it vanishing ideal\/} of $X$, denoted by
$I(X)$, is the ideal 
of $S$ generated by the homogeneous polynomials that vanish on $X$. 
It is known that $I(X)$ is a radical $1$-dimensional
Cohen-Macaulay lattice ideal
\cite{geramita-cayley-bacharach,algcodes}. For this class of ideals
not much is known about the complete intersection property. The
first result in this direction appears recently in \cite{ci-codes}, where it is shown
that if $X$ is parameterized by the edges of a simple hypergraph, 
then $I(X)$ is a complete intersection if and only if $X$ is a
projective torus.

The complete intersection property of $I(X)$ is relevant from the
viewpoint of algebraic coding theory as we now briefly explain. 
Roughly speaking, an {\it
evaluation code\/} over $X$ of degree $d$ is a linear space obtained
by evaluating all homogeneous $d$-forms of $S$ on
the set of points $X\subset{\mathbb P}^{s-1}$ (see
\cite{duursma-renteria-tapia,gold-little-schenck}). 
An evaluation code over
$X$ of degree $d$ has {\it length\/} $|X|$ and {\it dimension} 
$\dim_K (S/I(X))_d$. The main parameters (length, dimension, minimum
distance) of evaluation codes arising from complete 
intersections have been studied in
\cite{ballico-fontanari,duursma-renteria-tapia,gold-little-schenck,
hansen,ci-codes}. 

As a consequence of our results, in Section~\ref{vanishing-id-section}, 
we characterize the complete intersection property
of $I(X)$ in algebraic and geometric terms
(Corollary~\ref{ci-char-lattice}) and show that $I(X)$ is a binomial
set theoretic complete 
intersection (Corollary~\ref{stcib-i(x)}). 

For all unexplained
terminology and additional information,  we refer to 
\cite{EisStu,cca,monalg} (for the theory of binomial and lattice
ideals) and \cite{BHer} (for commutative algebra). 

\section{Complete intersections}\label{ci-lattice-section}

We continue to use the notation and definitions used in
Section~\ref{intro-ci-vanishing}. In this section we characterize the
complete intersection property of the family of all graded lattice 
ideals of dimension $1$. If ${\rm char}(K)>0$, we show that all
ideals of this family are  
binomial set theoretic complete intersections. In characteristic 
zero, we show that an arbitrary lattice ideal which is a binomial set
theoretic complete intersection is a complete intersection. 

Recall that a binomial in $S$ is a polynomial of the
form $t^a-t^b$, where 
$a,b\in \mathbb{N}^s$ and where, if
\mbox{$a=(a_1,\dots,a_s)\in\mathbb{N}^s$}, we set 
\[
t^a=t_1^{a_1}\cdots t_s^{a_s}\in S. 
\]
A binomial of the form $t^a-t^b$ is usually referred 
to as a {\it pure binomial\/} \cite{EisStu}, although here we are dropping the
adjective ``pure''.  A {\it binomial ideal\/} is an ideal generated by binomials. 

Given $c=(c_i)\in {\mathbb Z}^s$, we set ${\rm supp}(c)=\{i\, |\,
c_i\neq 0\}$. The set ${\rm supp}(c)$ is called the {\it support\/} of
$c$. The vector $c$ can be uniquely written as $c=c^+-c^-$, 
where $c^+$ and $c^-$ are two nonnegative vectors 
with disjoint support, the {\it positive\/} and 
the {\it negative\/} part of $c$ respectively. If $t^a$ is a monomial,
with $a=(a_i)\in\mathbb{N}^s$, we set ${\rm supp}(t^a)=\{t_i\vert\,
a_i>0\}$. If $f=t^a-t^b$ is a binomial, we set ${\rm supp}(f)={\rm
supp}(t^a)\cup{\rm supp}(t^b)$.

\begin{definition}\label{lattice-ideal-def}\rm 
Let $\mathcal{L}\subset \mathbb{Z}^s$ be a {\it lattice\/}, that is, 
$\mathcal{L}$ is a subgroup of $\mathbb{Z}^s$. The {\it lattice
ideal\/} of $\mathcal{L}$ is the binomial ideal
$$
I(\mathcal{L})=(\{t^{a^+}-t^{a^-}\vert\, 
a\in\mathcal{L}\})\subset S.
$$
\end{definition}

This concept is a natural generalization of a toric ideal
\cite[Corollary~7.1.4]{monalg}. Lattice ideals have been studied
extensively, see  \cite{EisStu,eto,cca} and the references there. 

The following is a well known description of lattice
ideals that follows from \cite[Corollary~2.5]{EisStu}.

\begin{theorem}{\rm\cite{EisStu}}\label{jun12-02} If $L$ is a binomial ideal 
of $S$, then $L$ is a lattice ideal if and only if $t_i$ is a non-zero
divisor of $S/L$ 
for all $i$. 
\end{theorem}

\begin{lemma}\label{jun13-02} If $a,\alpha_i,b,\beta_i$ are in 
$\mathbb{N}^s$ for $i=1,\ldots,r$ and $a-b$ is in the subgroup 
of $\mathbb{Z}^s$ generated by $\alpha_1-\beta_1,\ldots,\alpha_r-\beta_r$, 
then there is $t^\delta\in S$ such that
$$
t^\delta(t^a-t^b)\in (t^{\alpha_1}-t^{\beta_1},\ldots,t^{\alpha_r}-t^{\beta_r}).
$$
\end{lemma}

\begin{proof} We set $f=t^a-t^b$ and $g_i=t^{\alpha_i}-t^{\beta_i}$. There are 
integers $\lambda_1,\ldots,\lambda_r$ such that 
\[
(t^{a}/t^{b})-1=
\left(t^{\alpha_1}/t^{\beta_1}\right)^{\lambda_1}\cdots 
\left(t^{\alpha_r}/t^{\beta_r}\right)^{\lambda_r}-1.
\]
We may assume that $\lambda_i\geq 0$ for all $i$ by replacing, if necessary,
$t^{\alpha_i}/t^{\beta_i}$ by its inverse.  
Hence, writing 
$t^{\alpha_i}/t^{\beta_i}=((t^{\alpha_i}/t^{\beta_i})-1)+1$ 
and using the binomial theorem, it follows that 
$t^{\delta}f$ is in the ideal $(g_1,\ldots,g_r)$ for some monomial
$t^{\delta}$. 
\end{proof}

Let $\mathcal{G}$ be a subgroup of ${\mathbb Z}^s$. Following \cite{Eli}, we define
an equivalence relation $\sim_\mathcal{G}$ on the set of monomials of $S$ by
$t^{a}\sim_\mathcal{G} t^{b}$ if and only if $a-b\in \mathcal{G}$. 
A non-zero polynomial $f=\sum_{a}\lambda_a t^a$
is {\it simple\/} with respect to $\sim_\mathcal{G}$ if all its monomials
with non-zero coefficient are equivalent under $\sim_\mathcal{G}$.
An arbitrary non-zero polynomial $f$ in $S$ is uniquely expressed as:
$f=f_1+\cdots+f_m$ such that $f_i$ is simple and if $i\neq j$ and
$t^{a},t^b$ are monomials in $f_i$ and $f_j$ respectively,
then $t^a{\not\sim}_\mathcal{G}t^b$. We call $f_1,\ldots,f_m$ the 
{\it simple components\/} of $f$ with respect to $\sim_\mathcal{G}$.

Given a binomial $g=t^a-t^b$, we set $\widehat{g}=a-b$. If $B$ is a
subset of $\mathbb{Z}^s$, $\langle B\rangle$ denotes the subgroup of 
$\mathbb{Z}^s$ generated by $B$.

For convenience we recall the following result about the
behaviour of simple components.

\begin{lemma}{\cite[Lemma~2.2]{ElVi}}\label{sept-21-04} Let 
$I$ be a binomial ideal of $S$ generated by a set of binomials 
$g_1,\ldots,g_r$. If $0\neq f\in I$ and
$\mathcal{G}=\langle\widehat{g}_1,\ldots,\widehat{g}_r\rangle$, then 
any simple component of $f$ with respect to $\sim_\mathcal{G}$ belongs to $I$.
\end{lemma}

\begin{lemma}\label{dec21-11} Let $\mathcal{L}\subset\mathbb{Z}^s$ be
a lattice and let $I(\mathcal{L})$ be its lattice ideal. If
$g_1,\ldots,g_r$ is a set of binomials that generate
$I(\mathcal{L})$, then
$\mathcal{L}=\langle\widehat{g}_1,\ldots,\widehat{g}_r\rangle$. In
particular if $L$ is a lattice ideal, there is a unique lattice
$\mathcal{L}$ such that $L=I(\mathcal{L})$.
\end{lemma}

\begin{proof} Consider the lattice
$\mathcal{G}=\langle\widehat{g}_1,\ldots,\widehat{g}_r \rangle$. First
we show the inclusion $\mathcal{L}\subset\mathcal{G}$. Take
$0\neq a\in\mathcal{L}$. We can write $a=a^+-a^-$. Then,
$f=t^{a^+}-t^{a^-}$ is in $I(\mathcal{L})=(g_1,\ldots,g_r)$. By
Lemma~\ref{sept-21-04}, 
any simple component of $f$ with respect to $\sim_\mathcal{G}$ is also in
$(g_1,\ldots,g_r)$. Since $t^{a^+}$ and $t^{a^-}$ are not in
$I(\mathcal{L})$, then $f$ is a simple component of $f$ with respect
to $\sim_\mathcal{G}$, i.e., $a=a^+-a^-\in \mathcal{G}$. Thus,
$\mathcal{L}\subset\mathcal{G}$. To show the other inclusion notice
that a binomial $t^a-t^b$ is in $I(\mathcal{L})$ if and only if
$a-b\in\mathcal{L}$. This follows from Lemma~\ref{sept-21-04}. 
Hence, $a_i-b_i\in\mathcal{L}$ for all $i$, i.e., $\mathcal{G}\subset\mathcal{L}$.  
\end{proof}

The {\it affine space\/} of dimension $s$ over $K$, 
denoted by $\mathbb{A}_K^s$, is the cartesian 
product $K^s$ of $s$-copies of $K$. Given a 
subset $I\subset S$ its {\it zero set\/} or 
{\it variety\/}, denoted by $V(I)$, is the set of all 
$a\in\mathbb{A}_K^s$ such that $f(a)=0$ for all $f\in I$. 

\begin{lemma}\label{dec23-11} Let $I$ be a binomial ideal of $S$ such
that $V(I,t_i)=\{0\}$ for all $i$. If $\mathfrak{p}$ is a prime ideal
containing $(I,t_m)$ for some $1\leq m\leq s$, then
$\mathfrak{p}=(t_1,\ldots,t_s)$.
\end{lemma}

\begin{proof} Let $h_1,\ldots,h_r$ be a generating set for $I$
consisting of binomials. For simplicity of notation assume 
that $m=1$. We may assume that $t_1,\ldots,t_k$ are in $\mathfrak{p}$
and $t_i\notin\mathfrak{p}$ for $i>k$. If $t_i\in{\rm supp}(h_j)$ for
some $1\leq i\leq k$, say $h_j=t^{a_j}-t^{b_j}$ and $t_i\in {\rm
supp}(t^{a_j})$, then $t^{b_j}\in\mathfrak{p}$ and there is $1\leq
\ell\leq k$ such that $t_\ell$ is in the support of $t^{b_j}$. Thus, 
$h_j\subset(t_1,\ldots,t_k)$. Hence, for each $1\leq j\leq r$, either  
\begin{itemize}
\item[(i)\ ]${\rm supp}(h_j)\cap\{t_1,\ldots,t_k\}=\emptyset$ \mbox{ or}
\item[(ii)] $h_j\in (t_1,\ldots,t_k)$.
\end{itemize}
Consider the point $c=(c_i)\in\mathbb{A}_K^s$, with $c_i=0$ for $i\leq k$ and $c_i=1$
for $i>k$. If (i) occurs, then $h_j(c)=(t^{a_j}-t^{b_j})(c)=1-1=0$. 
If (ii) occurs, then $h_j(c)=(t^{a_j}-t^{b_j})(c)=0-0=0$. 
 Clearly the polynomial $t_1$ vanishes at $c$. Hence, $c\in
 V(I,t_1)=\{0\}$. Therefore, $k=s$. Thus, 
$\mathfrak{p}$ contains all the variables of $S$, i.e.,
$\mathfrak{p}=(t_1,\ldots,t_s)$.
\end{proof}

\begin{definition} Let $\omega=(\omega_1,\ldots,\omega_s)$ be an
integral vector with positive entries. 
A  lattice $\mathcal{L}$ is called {\it homogeneous\/} with respect to
$\omega$ if $\langle\omega,a\rangle=0$ for $a\in\mathcal{L}$. 
\end{definition}

A lattice $\mathcal{L}$ is homogeneous with respect to $\omega$ if and only if its lattice
ideal $I(\mathcal{L})$ is graded with respect to the grading
of $S$ induced by setting $\deg(t_i)=\omega_i$ for $i=1,\ldots,s$. The
standard grading of $S$ is obtained when $\omega=(1,\ldots,1)$. In
what follows by a graded ideal of $S$ we mean an ideal which is
graded with respect to the grading of $S$ induced by a vector $\omega$.

\begin{remark} If $\mathcal{L}$ is a homogeneous lattice in $\mathbb{Z}^s$ of
rank $s-1$, then $S/I(\mathcal{L})$ is a Cohen-Macaulay ring of
dimension $1$. This follows from Theorem~\ref{jun12-02} and using
the fact that the height of $I(\mathcal{L})$ is the rank of
$\mathcal{L}$ \cite[Proposition~7.5]{cca}.
\end{remark}

\begin{proposition}\label{apr24-12} Let $I\subset S$ be a graded binomial ideal. {\rm
(a)} If 
$V(I,t_i)=\{0\}$ for all $i$, then ${\rm ht}(I)=s-1$. 
{\rm (b)} If $I$ is a lattice ideal and ${\rm ht}(I)=s-1$,
then $V(I,t_i)=\{0\}$ for all $i$.
\end{proposition}

\begin{proof} (a) As $I$ is graded, all associated prime ideal of
$S/I$ are graded. Thus, all associated prime ideals of $S/I$ are
contained in $\mathfrak{m}=(t_1,\ldots,t_s)$. If ${\rm ht}(I)=s$,
then $\mathfrak{m}$ would be the only associated prime of
$S/I$, that is, $\mathfrak{m}$ is the radical of $I$, 
a contradiction because $I$ cannot contain a power of $t_i$ for any
$i$. Thus, ${\rm ht}(I)\leq s-1$. On the other hand, by Lemma~\ref{dec23-11},
the ideal $(I,t_s)$ has height $s$. Hence, $s={\rm ht}(I,t_s)\leq{\rm
ht}(I)+1$ (here we use the fact that $I$ is graded). 
Altogether, we get ${\rm ht}(I)=s-1$.

(b) Let $\mathcal{L}$ be the lattice that defines $I$ and let 
$g_1,\ldots,g_r$ be a generating set for
$I$ consisting of homogeneous binomials. By
Lemma~\ref{dec21-11}, one has the equality
$\mathcal{L}=\langle\widehat{g}_1,\ldots,\widehat{g}_r\rangle$. Notice
that $s-1={\rm ht}(I)={\rm rank}(\mathcal{L})$. Given
two distinct integers $1\leq i,k\leq s$, the vector space
$\mathbb{Q}^s$ is generated by
$e_k,\widehat{g}_1,\ldots,\widehat{g}_r$. Hence, as $\mathcal{L}$ is
homogeneous with respect to $\omega=(\omega_1,\ldots,\omega_s)$, there
are positive integers $r_i$ and $r_k$ such that
$r_ie_i-r_ke_k\in\mathcal{L}$ and $r_i\omega_i-r_k\omega_k=0$. By
Lemma~\ref{jun13-02}, there is $t^\delta$ such that
$t^\delta(t_i^{r_i}-t_k^{r_k})$ is in $I$. Hence, by
Theorem~\ref{jun12-02}, $t_i^{r_i}-t_k^{r_k}$ is in $I$. 
Therefore, $V(I,t_i)=\{0\}$ for $i=1,\ldots,s$. 
\end{proof}

\begin{example} Let $S=\mathbb{Q}[t_1,t_2,t_3]$. The ideal 
$I=(t_1^2-t_2t_3,t_1^2-t_1t_2)$ has height $2$ is not a lattice ideal 
and $V(I,t_1)\neq \{0\}$, that is, Proposition~\ref{apr24-12}(b) only holds for lattice ideals. 
\end{example}

\begin{definition} An ideal $I\subset S$ is called a {\it complete intersection\/} if 
there exists $g_1,\ldots,g_{r}$ in $S $ such that $I=(g_1,\ldots,g_{r})$, 
where $r$ is the height of $I$. 
\end{definition}

Recall that a graded ideal $I$ is a complete
intersection if and only if $I$ is generated by a homogeneous regular 
sequence with ${\rm ht}(I)$ elements (see \cite[Proposition~1.3.17, Lemma~1.3.18]{monalg}).

\begin{lemma}\label{apr24-12-1} Let $I\subset S$ be a graded
binomial ideal. If $V(I,t_i)=\{0\}$ for all $i$ and $I$ is
a complete intersection, then $I$ is a lattice ideal.
\end{lemma}

\begin{proof} By Proposition~\ref{apr24-12}(a), the height of $I$ is $s-1$.
It suffices to prove that $t_i$ is a non-zero divisor of $S/I$ for
all $i$ (see Theorem~\ref{jun12-02}). If $t_i$ is a
zero divisor of $S/I$ for some $i$, there is an associated prime 
ideal $\mathfrak{p}$ of $S/I$ containing $(I,t_i)$. Hence, using
Lemma~\ref{dec23-11}, we get 
that $\mathfrak{p}=\mathfrak{m}$, a contradiction because $I$ is a
complete intersection of height $s-1$ 
and all associated prime ideals of $I$ have height equal to $s-1$ (see
\cite[Proposition~1.3.22]{monalg}). 
\end{proof}

\begin{example} Let $S=\mathbb{Q}[t_1,t_2,t_3]$. The ideal 
$I=(t_1^2-t_2t_3,t_2^2-t_3^2)$ has height $2$ and $V(I,t_i)=\{0\}$
for all $i$. Thus, by Lemma~\ref{apr24-12-1}, $I$ is a lattice ideal. 
\end{example}

We come to one of the main results of this paper.

\begin{theorem}\label{ci-lattice} 
Let $L$ be the lattice ideal of a homogeneous lattice $\mathcal{L}$ in
$\mathbb{Z}^s$. If $V(L,t_i)=\{0\}$ for all $i$, then $L$ is a
complete intersection if and only 
if there are 
homogeneous binomials $h_1,\ldots,h_{s-1}$ in $L$ satisfying the following conditions\/{\rm :}
\begin{itemize}
\item[$(\mathrm{i})$]
$\mathcal{L}=\langle\widehat{h}_1,\ldots,\widehat{h}_{s-1}\rangle$.
\item[$(\mathrm{ii})$] $V(h_1,\ldots,h_{s-1},t_i)=\{0\}$ for all $i$.
\item[$(\mathrm{iii})$] $h_i=t^{a_i^+}-t^{a_i^-}$ for $i=1,\ldots,s-1$.
\end{itemize}
\end{theorem}

\begin{proof} As $\mathcal{L}$ is homogeneous, there is 
an integral vector $\omega=(\omega_1,\ldots,\omega_s)$ with positive entries such that
$\langle\omega,a\rangle=0$ for $a\in\mathcal{L}$. Then, its 
lattice ideal $L$ is graded with respect to the grading
of $S$ induced by setting $\deg(t_i)=\omega_i$ for $i=1,\ldots,s$. 
By Proposition~\ref{apr24-12}, the height of
$L$ is $s-1$.

$\Rightarrow$)  Since $L$ is a graded binomial ideal
which is a complete intersection, it is well known that
$L$ is an ideal generated by homogeneous binomials
$h_1,\ldots,h_{s-1}$ (see for instance 
\cite[Lemma~2.2.16]{monalg}). Then, by Lemma~\ref{dec21-11} and
Theorem~\ref{jun12-02}, (i) and (iii) hold. 
From the equality $(L,t_i)=(h_1,\ldots,h_{s-1},t_i)$,
we get
$$
\{0\}=V(L,t_i)=V(h_1,\ldots,h_{s-1},t_i). 
$$
Thus, $V(h_1,\ldots,h_{s-1},t_i)=\{0\}$ for all $i$, i.e., (ii) holds.

$\Leftarrow$) We set $I=(h_1,\ldots,h_{s-1})$. By 
hypothesis $I\subset L$. Thus, we need only show the inclusion 
$L\subset I$. Let $g_1,\ldots,g_m$ be a generating set 
of $L$ consisting of binomials, then $\widehat{g}_i\in\mathcal{L}$
for all $i$. Using condition (i) and 
Lemma~\ref{jun13-02}, for each $i$ there is a monomial $t^{\gamma_i}$ 
such that $t^{\gamma_i}g_i\in I$. Hence, $t^\gamma
L\subset I$, where $t^\gamma$ is equal to
$t^{\gamma_1}\cdots t^{\gamma_m}$. By (ii) and Proposition~\ref{apr24-12}, the height
of $I$ is $s-1$. This means that $I$ is a complete intersection. As
$t^\gamma L\subset I$, to show the inclusion 
$L\subset I$, it suffices to notice that by
(ii), Lemma~\ref{apr24-12-1} and Theorem~\ref{jun12-02} $t_i$ is a
non-zero divisor of $S/I$ for all $i$.  
\end{proof}
%hola

\begin{remark} The result remains valid if we remove condition (iii),
i.e., condition (iii) is redundant. In both implications of the
theorem the set $h_1,\ldots,h_{s-1}$ is shown to generate $L$.
\end{remark}

\begin{definition}
An ideal $I$ is called a {\it binomial set 
theoretic complete intersection\/}  if 
there are binomials $g_1,\ldots,g_{r}$ such that ${\rm rad}(I)={\rm
rad}(g_1,\ldots,g_{r})$, where $r$ is the height of $I$. 
\end{definition}

The next result gives a family of binomial set 
theoretic complete intersections. 
We show this result using a theorem of Katsabekis, Morales and 
Thoma \cite[Theorem~4.4(2)]{katsabekis-morales-thoma}.

\begin{proposition}\label{stcib-1-dim-lattice} If $K$ is a field of
positive characteristic and $L\subset S$ is a graded
lattice ideal of dimension $1$, then $L$ is a binomial set theoretic complete
intersection.
\end{proposition}

\begin{proof} Let $\mathcal{L}$ be the homogeneous lattice of
$\mathbb{Z}^s$ such that $L=I(\mathcal{L})$. Notice that $\mathcal{L}$ is
a lattice of rank $s-1$ because ${\rm ht}(L)={\rm
rank}(\mathcal{L})$. Thus, there is an isomorphism of groups 
$\psi\colon\mathbb{Z}^s/{\rm Sat}(\mathcal{L})\rightarrow\mathbb{Z}$,
where ${\rm Sat}(\mathcal{L})$ is the saturation of $\mathcal{L}$
consisting of all $a\in\mathbb{Z}^s$ such that $da\in\mathcal{L}$ for
some $0\neq d\in\mathbb{Z}$. For each $1\leq i\leq s$, we set
$a_i=\psi(e_i+{\rm Sat}(\mathcal{L}))$, where $e_i$ is the $i${\it th}
unit vector in $\mathbb{Z}^s$. Following \cite{katsabekis-morales-thoma}, 
the multiset $A=\{a_1,\ldots,a_s\}$ is called the configuration of
vectors associated to $\mathcal{L}$. Recall 
that $s-1={\rm rank}(\mathcal{L})$. Hence, as $\mathcal{L}$ is
homogeneous with respect to $\omega=(\omega_1,\ldots,\omega_s)$, there
are positive integers $r_i$ and $r_k$ such that
$r_ie_i-r_ke_k\in\mathcal{L}$ and $r_i\omega_i-r_k\omega_k=0$. Thus,
$r_ia_i=r_ka_k$ and $a_i$ has the same sign as $a_k$. This means
that $a_1,\ldots,a_s$ are all positive or all negative.  It follows
that $A$ is a full configuration in the sense  
of \cite[Definition~4.3]{katsabekis-morales-thoma}. Thus,
$I(\mathcal{L})$ is a binomial set theoretic complete intersection by
\cite[Theorem~4.4(2)]{katsabekis-morales-thoma} and its proof.
\end{proof}

\begin{corollary}{\cite{stcib}}\label{LaConcepcion-dec24-2011} 
If $P\subset S$ is the toric ideal of a monomial curve, then 
$P$ is a complete intersection if and only if there are
homogeneous binomials $g_1,\ldots,g_{s-1}$ in $P$, with
$g_i=t^{a_i^+}-t^{a_i^-}$ for all $i$, such that the following
conditions hold\/{\rm :}
\begin{itemize}
\item[$(\mathrm{a})$] $\mathcal{L}_1=\langle
\widehat{g}_1,\ldots,\widehat{g}_{s-1}\rangle$, where
$\mathcal{L}_1$ is the lattice that defines $P$. 
\item[$(\mathrm{b})$] $V(g_1,\ldots,g_{s-1},t_i)=\{0\}$ for $i=1,\ldots,s$.
\end{itemize}
\end{corollary}

\begin{proof} There are positive integers $\omega_1,\ldots,\omega_s$
such that $P$ is the kernel of 
the epimorphism of $K$-algebras: 
$$\varphi\colon K[t_1,\ldots,t_s]\longrightarrow
K[y_1^{\omega_1},\ldots,y_1^{\omega_s}],\ \ \ \ \ 
f\stackrel{\varphi}{\longmapsto}
f(y_1^{\omega_1},\ldots,y_1^{\omega_s}),$$ 
where $y_1$ is a new variable. Consider the homomorphism of $\mathbb{Z}$-modules
$\psi\colon\mathbb{Z}^s\rightarrow\mathbb{Z}$, $e_i\mapsto
\omega_i$. According to \cite[Corollary~7.1.4]{monalg}, the toric ideal $P$ is
the lattice ideal 
of the homogeneous lattice
$\mathcal{L}_1={\rm ker}(\psi)$ with respect to the vector
$\omega=(\omega_1,\ldots,\omega_s)$, that is $P=I(\mathcal{L}_1)$. In
particular the height of $P$ is $s-1$. The binomial
$t_i^{\omega_j}-t_j^{\omega_i}$ is in $P$ for all $i,j$. Thus, 
$V(I(\mathcal{L}_1),t_i)=\{0\}$ for all $i$. Then, the result follows 
from Theorem~\ref{ci-lattice}.
\end{proof}

\begin{corollary}{\rm\cite{moh}}\label{stcib-p} Let $P\subset S$ be the toric
ideal of a monomial curve. If ${\rm char}(K)>0$,
then $P$ is 
a binomial set theoretic complete
intersection.
\end{corollary}

\begin{proof} As seen in the proof of
Corollary~\ref{LaConcepcion-dec24-2011}, 
$P$ is a $1$-dimensional graded lattice ideal. 
Thus, the result follows 
at once from Proposition~\ref{stcib-1-dim-lattice}. 
\end{proof}

We come to another of our main results.

\begin{theorem}\label{adelantado1} Let $L\subset S$ be an arbitrary lattice ideal of
height $r$. If ${\rm char}(K)=0$ and  ${\rm rad}(L)={\rm rad}(g_1,\ldots,g_r)$ for 
some binomials $g_1,\ldots,g_r$, then 
$L=(g_1,\ldots,g_r)$. 
\end{theorem}

\begin{proof} Consider the binomial ideal $I=(g_1,\ldots,g_r)$, where
$g_i=t^{a_i}-t^{b_i}$ for $i=1,\ldots,r$. Since ${\rm rad}(I)$ is again a binomial ideal
(see \cite[Theorem~9.4 and Corollary~9.12]{gilmer}), we may assume
that ${\rm rad}(I)$ is generated 
by a set of binomials $\{h_1,\ldots,h_m\}$. 
From \cite[Corollary~9.12, p.~106]{gilmer}, it is seen that any lattice
ideal over a field $K$ of characteristic zero is radical, i.e., 
${\rm rad}(L)=L$. Let 
\begin{equation}\label{apr28-12-0}
I=\mathfrak{q}_1\cap\cdots\cap \mathfrak{q}_p
\end{equation}
be a primary decomposition of $I$. Since $I$ is an ideal of height $r$
generated by $r$ elements and $S$ is Cohen-Macaulay, by the unmixedness theorem
\cite[Theorem~2.1.6]{BHer}, $I$ has no 
embedded primes. Hence, ${\rm rad}(\mathfrak{q}_i)=\mathfrak{p}_i$ 
is a minimal prime of both $I$ and $L$ for $i=1,\ldots,p$. Since ${\rm char}(K)=0$, by
\cite[Lemma~2.2]{stcib}, we 
have the equality 
\begin{equation}\label{apr28-12}
\langle\widehat{g}_1,\ldots,\widehat{g}_r\rangle=
\langle\widehat{h}_1,\ldots,\widehat{h}_m\rangle.
\end{equation}
The inclusion $I\subset L$ is clear. We now show the reverse
inclusion. 
Take a binomial $h$ in $L$. Since $L$ is generated by
$h_1,\ldots,h_m$, by Lemma~\ref{dec21-11}, the lattice that
defines $L$ is 
$\langle\widehat{h}_1,\ldots,\widehat{h}_m\rangle$. 
 Therefore, using Eq.~(\ref{apr28-12}) and
Lemma~\ref{jun13-02}, we get that there is
a monomial $t^\alpha$ so that $t^\alpha{h}\in I$. Thus, by
Eq.~(\ref{apr28-12-0}), $t^\alpha{h}\in\mathfrak{q}_i$ for all $i$. If $t^\alpha=1$,
then $h\in I$ and there is nothing to prove. Assume that $t^\alpha\neq
1$.  It suffices to prove that
$h$ belongs to $\mathfrak{q}_i$ for all $i$. If
$h\notin\mathfrak{q}_i$ for some $i$, then
$(t^\alpha)^\ell\in\mathfrak{q}_i$ and consequently $\mathfrak{p}_i$
must contain at least one variable $t_k$. Since $\mathfrak{p}_i$ is a minimal prime of
$L$, all its elements are zero divisors of $S/L$. In particular $t_k$
must be a zero divisor of $S/L$, a contradiction 
because $L$ is a lattice ideal and none of the variables of $S$ can
be a zero divisor of $S/L$ (see Theorem~\ref{jun12-02}). 
\end{proof}

As a consequence, we recover the following result.

\begin{corollary}{\cite{BMT}}\label{toric-version} Let $P\subset S$ be an
arbitrary toric ideal of height $r$. If ${\rm char}(K)=0$ and  $P={\rm rad}(g_1,\ldots,g_r)$ for 
some binomials $g_1,\ldots,g_r$, then $P=(g_1,\ldots,g_r)$. 
\end{corollary}

\section{Vanishing ideals over finite
fields}\label{vanishing-id-section}

We continue to use the notation and definitions used in
Sections~\ref{intro-ci-vanishing} and \ref{ci-lattice-section}. 
In this section we characterize the complete intersection
property of vanishing ideals over algebraic toric sets parameterized by
monomials.  Throughout this section we assume that the polynomial ring $S$ has the
standard grading induced by the vector $\mathbf{1}=(1,\ldots,1)$. 

\begin{lemma}\label{lemma-homogeneous-lattice} If $\mathbb{F}_q$ is a finite
field, then there is a unique homogeneous
lattice $\mathcal{L}$ with respect to the vector
$\omega=(1,\ldots,1)$ such that $I(X)=I(\mathcal{L})$. 
\end{lemma}

\begin{proof} By \cite[Theorem~2.1]{algcodes}, $I(X)$ is lattice ideal 
generated by homogeneous binomials. Let $\mathcal{L}$ be a homogeneous lattice
that defines $I(X)$. The uniqueness of $\mathcal{L}$ follows from
Lemma~\ref{dec21-11}. 
\end{proof}

\begin{corollary}\label{ci-char-lattice} If $\mathbb{F}_q$ is a finite
field, then $I(X)$ is a complete intersection if and only if there are 
homogeneous binomials $h_1,\ldots,h_{s-1}$ in $I(X)$ such that the following 
conditions hold\/{\rm :}
\begin{itemize}
\item[$(\mathrm{i})$]
$\mathcal{L}=\langle\widehat{h}_1,\ldots,\widehat{h}_{s-1}\rangle$,
where $\mathcal{L}$ is the lattice that defines $I(X)$. 
\item[$(\mathrm{ii})$] $V(h_1,\ldots,h_{s-1},t_i)=\{0\}$ for
$i=1,\ldots,s$.
\item[$(\mathrm{iii})$] $h_i=t^{a_i^+}-t^{a_i^-}$ for
$i=1,\ldots,s-1$.
\end{itemize}
\end{corollary}

\begin{proof} By Lemma~\ref{lemma-homogeneous-lattice}, there is a
unique homogeneous lattice $\mathcal{L}$ with respect to the vector
$\omega=\mathbf{1}$ such that $I(X)=I(\mathcal{L})$. The
binomial $t_i^{q-1}-t_j^{q-1}$ is in $I(X)$ for all $i,j$. Thus, 
$V(I(\mathcal{L}),t_i)=\{0\}$ for all $i$. Therefore the result
follows from Theorem~\ref{ci-lattice}.
\end{proof}

\begin{corollary}\label{stcib-i(x)} If $\mathbb{F}_q$ is a finite
field, then $I(X)$ is a binomial set theoretic complete
intersection.
\end{corollary}

\begin{proof} $I(X)$ is a $1$-dimensional graded lattice ideal 
\cite{geramita-cayley-bacharach,algcodes}. Thus, the result follows 
at once from Proposition~\ref{stcib-1-dim-lattice}.
\end{proof}

\bibliographystyle{plain}

\end{document}